\documentclass[11pt]{amsart}
\usepackage[margin=1in]{geometry}
\usepackage{latexsym}
\usepackage{amsfonts}
\usepackage{amsmath}
\usepackage{amssymb}
\usepackage{amsthm}
\usepackage{enumerate}
\usepackage[hang,flushmargin]{footmisc}
\usepackage{caption}
\usepackage{tabu}
\usepackage{mathrsfs}
\usepackage{amsaddr}

\usepackage{graphicx}
\usepackage{epstopdf}
\usepackage{epsfig}
\usepackage{caption}
 
\usepackage{bm} 
\usepackage{tikz}

\newtheorem{theorem}{Theorem}[section]
\newtheorem{lemma}{Lemma}[section]

\newtheorem{corollary}{Corollary}[section]

\newtheorem{conjecture}{Conjecture}[section]

\theoremstyle{definition}

\newtheorem{remark}{Remark}[section]
\newtheorem{example}{Example}[section]

\begin{document}
\title{Poset Pattern-Avoidance Problems Posed by Yakoubov}
\author{Colin Defant}
\address{University of Florida \\ 1400 Stadium Rd. \\ Gainesville, FL 32611 United States}
\email{cdefant@ufl.edu}

\begin{abstract} 
Extending the work of Yakoubov, we enumerate the linear extensions of comb posets that avoid certain length-$3$ patterns. We resolve many of Yakoubov's open problems and prove both of the conjectures from her paper. 
\end{abstract} 

\maketitle

\bigskip


 
\section{Introduction} 
A linear extension of a partially ordered set $P$ is an order-preserving bijection from $P$ to a totally ordered set. We may view a linear extension as a permutation $\pi$ of the elements of $P$ with the property that $x$ precedes $y$ in $\pi$ whenever $x<_Py$. Determining $e(P)$, the number of linear extensions of the poset $P$, is a problem that has motivated an enormous amount of research in order theory. Indeed, Richard Stanley \cite[p. 258]{Stanley} states that $e(P)$ ``is probably the single most useful number for measuring the `complexity' of $P$." The problem of determining $e(P)$ for various posets $P$ is of crucial importance in the study of sorting algorithms, particularly when one has incomplete information concerning the objects that need to be sorted. Brightwell and Winkler showed that, in general, this problem is $\#P$-complete \cite{Brightwell}. However, several authors have developed algorithms to perform tasks such as quickly approximating $e(P)$ or generating random linear extensions of $P$ \cite{Banks, De Loof, Karzanov}. 

Another field of study that has spawned a great deal of research is that of permutation pattern avoidance. Suppose $\pi=\pi_1\pi_2\cdots\pi_n\in S_n$ and $\tau=\tau_1\tau_2\cdots\tau_k\in S_k$. We say that $\pi$ \emph{contains the pattern $\tau$} if there are indices $i_1,i_2,\ldots,i_k\in[n]$ with $i_1<i_2<\cdots <i_k$ such that for all $j,\ell\in[k]$, we have $\pi_{i_j}<\pi_{i_\ell}$ if and only if $\tau_j<\tau_\ell$. If $\pi$ does not contain the pattern $\tau$, we say that $\pi$ \emph{avoids} $\tau$. Mikl\'os B\'ona's book \emph{Combinatorics of Permutations} provides an excellent reference for anyone wishing to learn more about the flourishing area of research that deals with permutation patterns \cite{Bona}. 

Recently, Sophia Yakoubov posed a natural question that links the study of linear extensions of posets with that of permutation patterns \cite{Yakoubov}. Suppose we are given a finite poset $P$ with $\vert P\vert=n$. In addition, suppose we bijectively label the elements of $P$ with the elements of $[n]$. Using this labeling, we may view each linear extension of $P$ as a permutation in $S_n$. Given $\tau\in S_k$, we wish to compute $A_\tau(P)$, the number of linear extensions of $P$ that avoid $\tau$. As stated, this is a very general and difficult problem. Therefore, Yakoubov devoted her attention to a specific class of posets called \emph{combs}. Anderson, Egge, Riehl, Ryan, Steinke, and Vaughan recently extended this line of research by investigating pattern-avoiding linear extensions of ``rectangular posets" \cite{Anderson}.  

The comb $K_{s,t}$ is specified by two parameters: the spine length $s$ and the tooth length $t$. The base set for this poset is $\{e_{i,j}\colon 1\leq i\leq s, 1\leq j\leq t\}$. We define the partial ordering on $K_{s,t}$ by specifying the covering relations as follows:
\begin{enumerate}[(i)]
\item $e_{i,1}\lessdot e_{i+1,1}$ for all $1\leq i\leq s-1$;
\item $e_{i,j}\lessdot e_{i,j+1}$ for all $1\leq i\leq s$ and $1\leq j\leq t-1$.   
\end{enumerate} 
The comb $K_{4,3}$ is displayed as the leftmost image in Figure \ref{Fig1}. 

Yakoubov considers two canonical ways to label the elements of $K_{s,t}$ with the elements of $[st]$. In the first labeling, called an $\alpha$-labeling, we give the element $e_{i,j}$ the label $(j-1)s+i$. In the second labeling, called a $\beta$-labeling, we give $e_{i,j}$ the label $(i-1)t+j$. We denote the $\alpha$-labeled (respectively, $\beta$-labeled) comb $K_{s,t}$ by $K_{s,t}^\alpha$ (respectively, $K_{s,t}^\beta$). The middle and right images in Figure \ref{Fig1} portray the labeled posets $K_{4,3}^\alpha$ and $K_{4,3}^\beta$, respectively.  

\begin{figure}[t]
\begin{center} 
\includegraphics[height=4cm]{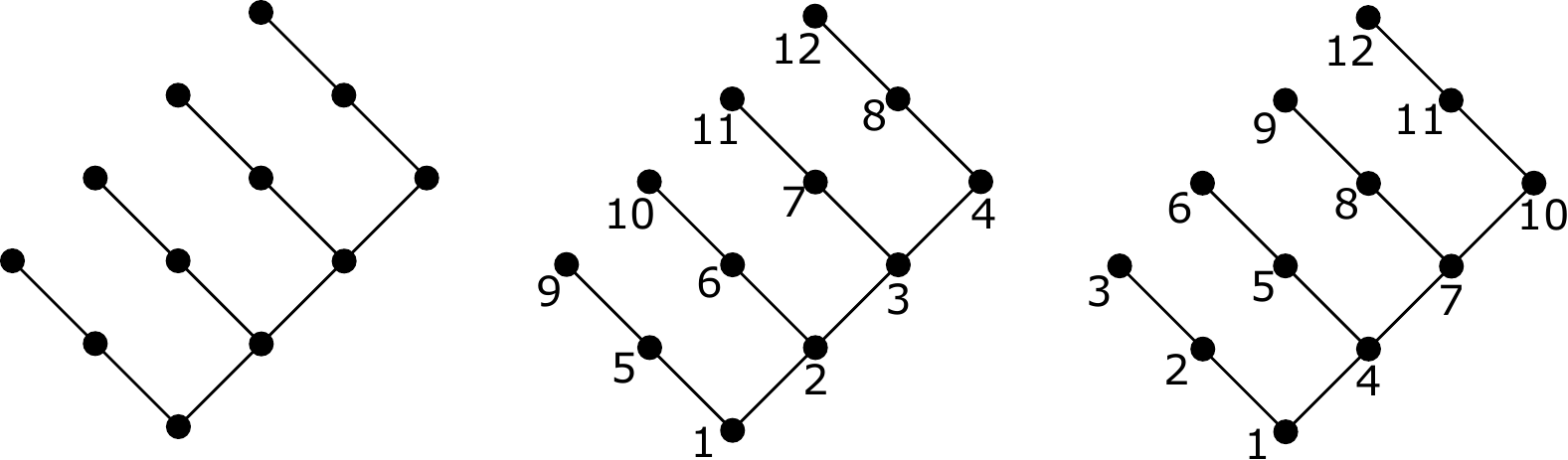}
\end{center}
\captionof{figure}{The leftmost figure is the Hasse diagram of the comb $K_{4,3}$. The image in the center is the Hasse diagram of the $\alpha$-labeled comb $K_{4,3}^\alpha$. On the right is the Hasse diagram of $K_{4,3}^\beta$.} \label{Fig1}
\end{figure}

Yakoubov enumerated the linear extensions of $K_{s,t}^\alpha$ and $K_{s,t}^\beta$ that avoid various patterns of length $3$. However, she did not enumerate these linear extensions for all choices of patterns. In the following section, we summarize the results, open problems, and conjectures given in Yakoubov's paper. In the third section, we resolve many of those open problems and conjectures. We end with some suggestions for further research. 

\section{Yakoubov's Results, Open Problems, and Conjectures} 
In this section, we summarize the paper \cite{Yakoubov}, omitting all proofs. Recall the definitions of the labeled combs $K_{s,t}^\alpha$ and $K_{s,t}^\beta$ from the introduction. When applying inductive arguments to these labeled posets, one often needs to build large combs from smaller combs by annexing one element at a time. Therefore, we will be working with what Yakoubov calls ``uneven combs." The labeled uneven comb $U_{\text{spine}=s,n}^\alpha$ is the labeled subposet of $K_{s,\left\lceil n/s\right\rceil}^\alpha$ obtained by removing all elements whose labels are larger than $n$ (so $\vert U_{\text{spine}=s,n}^\alpha\vert=n$). Figure \ref{Fig2} shows the labeled uneven comb $U_{\text{spine}=5,12}^\alpha$.

\begin{figure}[t]
\begin{center} 
\includegraphics[height=4cm]{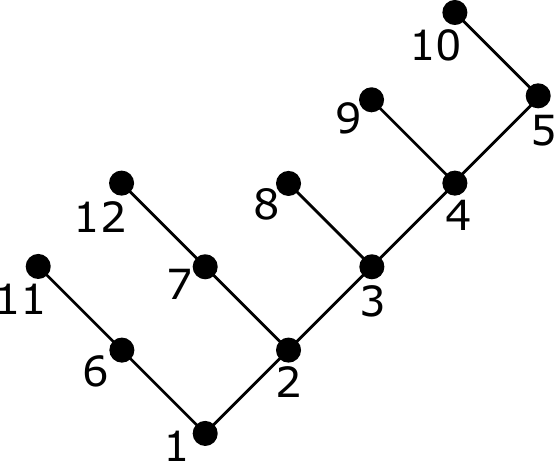}
\end{center}
\captionof{figure}{The Hasse diagram of the uneven comb $U_{\text{spine}=s,12}^\alpha$.} \label{Fig2}
\end{figure}

Recall that $A_{\tau}(P)$ is the number of linear extensions of the labeled poset $P$ that avoid the pattern $\tau$. More generally, denote by $A_{\tau_1,\tau_2,\ldots,\tau_r}(P)$ the number of linear extensions of $P$ avoiding all of the patterns $\tau_1,\tau_2,\ldots,\tau_r$. 

\begin{theorem}[Yakoubov, \cite{Yakoubov}]\label{ThmY1}
Let $C_m=\frac{1}{m+1}{2m\choose m}$ be the $m^\text{th}$ Catalan number. We have 
\begin{itemize}
\item $A_{123}(K_{s,t}^\alpha)=0$ if $s,t\geq 2$; 
\item $A_{132}(K_{s,t}^\alpha)=1$; 
\item $A_{213}(K_{s,t}^\alpha)=C_s$ if $t>1$;
\item $A_{312}(K_{s,2}^\alpha)=C_{s+1}-C_s$.  
\end{itemize}
\end{theorem}
Yakoubov leaves the determination of $A_{312}(K_{s,t}^\alpha)$ for $t>2$ as an open problem. She also leaves open the determination of $A_{231}(K_{s,t}^\alpha)$ and $A_{321}(K_{s,t}^\alpha)$ and makes the following conjecture. 
\begin{conjecture}[Yakoubov, \cite{Yakoubov}]\label{Conj1}
We have \[\sum_{s\geq 0}A_{231}(K_{s,2}^\alpha)x^s=C(x\cdot C(x)),\] where $C(x)=\dfrac{1-\sqrt{1-4x}}{2x}$ is the generating function of the Catalan numbers. 
\end{conjecture}

In Section 3, we show that \[\sum_{s\geq 0}A_{231}(U_{\text{spine}=s,2s-h}^\alpha)x^s=(x\cdot C(x))^hC(x\cdot C(x))\] for each $h\geq 0$. When $h=0$, this is exactly Conjecture \ref{Conj1} because $U_{\text{spine}=s,2s}^\alpha=K_{s,2}^\alpha$. We also derive simple recurrence relations for $A_{231}(U_{\text{spine}=s,n}^\alpha)$ and $A_{312}(U_{\text{spine}=s,n}^\alpha)$. This settles the problems of determining $A_{231}(K_{s,t}^\alpha)$ and $A_{312}(K_{s,t}^\alpha)$. Furthermore, we derive a method for computing $A_{321}(K_{s,2}^\alpha)$ and show that \[4^{t-2}(3+2\sqrt 2)\leq\lim_{s\to\infty}\sqrt[s]{A_{321}(K_{s,t}^\alpha)}\leq\frac{(2t-1)^{2t-1}}{(t-1)^{t-1}t^t}\] for each fixed $t\geq 2$. 

The next theorem concerns linear extensions of $\beta$-labeled combs. 
\begin{theorem}[Yakoubov, \cite{Yakoubov}]\label{ThmY2} 
We have 
\begin{itemize}
\item $A_{123}(K_{s,t}^\beta)=0$ if $s,t\geq 2$;
\item $A_{132}(K_{s,t}^\beta)=1$; 
\item $A_{213}(K_{s,t}^\beta)=A_{231}(K_{s,t}^\beta)=t^{s-1}$;
\item $A_{312}(K_{s,t}^\beta)=\frac{1}{st+1}{s(t+1)\choose s}$ if $t>1$;
\item $A_{321}(K_{s,2}^\beta)=\frac{1}{2s+1}{3s\choose s}$.
\end{itemize}
\end{theorem} 

Yakoubov leaves the enumeration of the $321$-avoiding linear extensions of $K_{s,t}^\beta$ for $t\geq 3$ as an open problem. In Section 3, we define numbers $F_{i,j}(k)$ recursively and give a formula for $A_{321}(K_{s,t}^\beta)$ involving the numbers $F_{i,j}(k)$. We then show that \[\lim_{t\to\infty}\sqrt[t]{A_{321}(K_{s,t}^\beta)}=2^s\] for each fixed $s\geq 2$. 

The following theorem focuses on the enumeration of linear extensions of combs that avoid two different length-$3$ patterns. 
\begin{theorem}[Yakoubov, \cite{Yakoubov}]\label{ThmY3}
If $s,t\geq 2$ and $\tau\in S_3$, then $A_{123,\tau}(K_{s,t}^\alpha)=A_{123,\tau}(K_{s,t}^\beta)=0$. If $\tau\in\{213,231,312,321\}$, then $A_{132,\tau}(K_{s,t}^\alpha)=A_{132,\tau}(K_{s,t}^\beta)=1$. Furthermore, 
\begin{itemize}
\item $A_{213,231}(K_{s,t}^\alpha)=A_{213,312}(K_{s,t}^\alpha)=A_{213,312}(K_{s,t}^\beta)=A_{231,312}(K_{s,t}^\beta)=2^{s-1}$; 
\item $A_{213,321}(K_{s,t}^\alpha)={s\choose 2}+1$; 
\item \begin{equation*}
\begin{split}
A_{312,321}(U_{\text{spine}=s,n}^\alpha) & = A_{231,312}(U_{\text{spine}=s,n}^\alpha) \\
 & =\begin{cases} 1, & \mbox{if } n\leq s; \\ 2^{n-s}, & \mbox{if } s<n<2s; \\ 2A_{231,312}(U_{\text{spine}=s,n-1}^\alpha)-A_{231,312}(U_{\text{spine}=s,n-1-s}^\alpha), & \mbox{if } 2s\leq n; \end{cases}
\end{split}
\end{equation*} 
\item $A_{213,231}(K_{s,t}^\beta)=1$;
\item $A_{213,321}(K_{s,t}^\beta)=(s-1)(t-1)+1$;
\item $A_{231,321}(K_{s,t}^\beta)=t^{s-1}$;
\item $A_{312,321}(K_{s,t}^\beta)=(t+1)^{s-1}$.

\end{itemize}  
\end{theorem}
Yakoubov did not enumerate the linear extensions of $K_{s,t}^\alpha$ avoiding $231$ and $321$. She did, however, make the following conjecture. 
\begin{conjecture}[Yakoubov, \cite{Yakoubov}]\label{Conj2}
For all $s\geq 1$, $A_{231,321}(K_{s,2}^\alpha)=2^{s-1}(s-1)+1$. 
\end{conjecture}
In the following section, we prove that $A_{231,321}(U_{\text{spine}=s,n}^\alpha)=\sum_{j=1}^sA_{231,321}(U_{\text{spine}=s,n-j}^\alpha)$ for all $n>s$. Conjecture \ref{Conj2} will then follow as an immediate consequence. 

\section{New Enumerative Results}
We now settle many of the open problems that Yakoubov proposed. Let $C_m=\frac{1}{m+1}{2m\choose m}$ denote the $m^\text{th}$ Catalan number, and let $C(x)=\dfrac{1-\sqrt{1-4x}}{2x}=\displaystyle{\sum_{m\geq 0}C_mx^m}$. 
\begin{center}
\underline{\bf 3.1. $A_{231}(K_{s,t}^\alpha)$}
\end{center}

To begin, let us enumerate $231$-avoiding linear extensions of $U_{\text{spine}=s,n}^\alpha$. When decomposing these linear extensions into simpler parts, we will obtain $231$-avoiding permutations that start with their largest entries and have the entry $1$ in a fixed position. Therefore, we need the following lemma. 

\begin{lemma}\label{Lem1}
Fix $k\geq 1$. Let $X_0(k,m)$ be the set of $231$-avoiding permutations $p=p_1p_2\cdots p_m\in S_m$ such that $p_1=m$ and $p_k=1$. We have \[\sum_{m\geq 1}\vert X_0(k,m)\vert x^m=x^kC(x)^{k-1}.\]
\end{lemma}  
\begin{proof}
The lemma is obvious when $k=1$, so assume $k\geq 2$. Let $u(\ell,n)$ be the number of $123$-avoiding permutations in $S_n$ in which the entry $n$ appears in the $\ell^\text{th}$ position. For $\ell\geq 1$, Connolly, Gabor, and Godbole have shown that $\sum_{n\geq 0}u(\ell,n)x^n=x^\ell C(x)^\ell$ \cite[Corollary 1]{Connolly}. We will show that $\vert X_0(k,m)\vert=u(k-1,m-1)$, and the lemma will follow immediately. 

A \emph{left-to-right minimum} of a permutation $\pi=\pi_1\pi_2\cdots\pi_n$ is an entry $\pi_i$ such that $\pi_i<\pi_j$ for all $j<i$. Simion and Schmidt constructed a specific bijection $\psi$ (now called the Simion-Schmidt bijection) from the set of $123$-avoiding permutations in $S_n$ to the set of $132$-avoiding permutations in $S_n$ \cite{Simion}. The Simion-Schmidt bijection has the additional property that it fixes the values and the positions of the left-to-right minima of its argument. In particular, the position of $1$ in $\pi$ is the same as the position of $1$ in $\psi(\pi)$. 

For each $\pi=\pi_1\pi_2\cdots\pi_n\in S_n$, let $\pi^c$ be the complement of $\pi$. That is, the $i^\text{th}$ entry of $\pi^c$ is $n+1-\pi_i$. Furthermore, let $\pi^r=\pi_n\pi_{n-1}\cdots\pi_1$ be the reverse of $\pi$.    By simply removing the largest entry of a permutation in $X_0(k,m)$, we see that $\vert X_0(k,m)\vert$ is equal to the number of $231$-avoiding permutations in $S_{m-1}$ whose $(k-1)^\text{th}$ entry is $1$. 
Observe that $\pi$ is a $123$-avoiding permutation in $S_{m-1}$ whose $(k-1)^\text{th}$ entry is $m-1$ if and only if $(\psi((\pi^c)^r))^r$ is a $231$-avoiding permutation in $S_{m-1}$ whose $(k-1)^\text{th}$ entry is $1$. It follows that $\vert X_0(k,m)\vert=u(k-1,m-1)$.
\end{proof}
\begin{theorem}\label{Thm1}
For each $h\geq 0$, \[\sum_{s\geq 0}A_{231}(U_{\text{spine}=s,2s-h}^\alpha)x^s=(x\cdot C(x))^hC(x\cdot C(x)).\] For $n\geq 2s$, we have $A_{231}(U_{\text{spine}=s,n}^\alpha)=\displaystyle{\sum_{j=1}^sC_{j-1}A_{231}(U_{\text{spine}=s,n-j}^\alpha)}$.  
\end{theorem}
\begin{proof}
Fix $h\geq 0$, and suppose $s\geq h+1$. Let $X_s(k,m)$ be the set of $231$-avoiding permutations of the elements of $\{s+1,s+2,\ldots,s+m\}$ whose first entry is $s+m$ and whose $k^{\text{th}}$ entry is $s+1$. Let $Y_{s,h}(m)$ be the set of $231$-avoiding permutations of $\{s+m+1,s+m+2,\ldots,2s-h\}$. Let $Z_s(k,m)$ be the set of $(s-m)$-element subsets of $\{m+1,m+2,\ldots,s+k\}$. We will give a bijection $\varphi$ from the set of $231$-avoiding linear extensions of $U_{\text{spine}=s,2s-h}^\alpha$ to $\displaystyle{\bigcup_{k\geq 1}\bigcup_{m\geq 1}}(X_s(k,m)\times Y_{s,h}(m)\times Z_s(k,m))$.

Let $\pi=\pi_1\pi_2\cdots\pi_{2s-h}$ be a $231$-avoiding linear extension of $U_{\text{spine}=s,2s-h}^\alpha$. If we remove the entries $1,2,\ldots,s$ from $\pi$, we obtain a $231$-avoiding permutation $\sigma=\sigma_1\sigma_2\cdots\sigma_{s-h}$ of the elements of $\{s+1,\ldots,2s-h\}$. Suppose $\sigma_1=s+m$ and $\sigma_k=s+1$. It follows from the fact that $\sigma$ avoids $231$ that $k\leq m$. Moreover, $\sigma_1\sigma_2\cdots\sigma_m\in X_s(k,m)$ and $\sigma_{m+1}\sigma_{m+2}\cdots\sigma_{s-h}\in Y_{s,h}(m)$. 

The entries of $\sigma_1\sigma_2\cdots\sigma_k$ must appear in decreasing order since $\sigma$ avoids $231$ and $\sigma_k=s+1$ is the smallest entry in $\sigma$. Because $\pi$ is a linear extension of $U_{\text{spine}=s,2s-h}^\alpha$, $m$ precedes $s+m=\sigma_1$ in $\pi$. We also know that the elements of $[s]$ must appear in increasing order in $\pi$. This means that $\pi_1\pi_2\cdots\pi_m=12\cdots m$. Observe that the entry $s$ must precede $\sigma_{k+1}$ in $\pi$ since otherwise, $\sigma_k\sigma_{k+1}s$ would form a $231$ pattern in $\pi$. This means that $\pi_{s+k+1}\pi_{s+k+2}\cdots\pi_{2s-h}=\sigma_{k+1}\sigma_{k+2}\cdots\sigma_{s-h}$. We deduce that $\pi_{m+1}\pi_{m+2}\cdots\pi_{s+k}$ is a permutation of $\{m+1,m+2,\ldots,s\}\cup \{\sigma_1,\sigma_2,\ldots,\sigma_k\}$. Let $T=\{j\in\{m+1,m+2,\ldots,s+k\}\colon \pi_j\in\{m+1,m+2,\ldots,s\}\}$. Note that $T\in Z_s(k,m)$. We define $\varphi(\pi)=(\sigma_1\sigma_2\cdots\sigma_m,\sigma_{m+1}\sigma_{m+2}\cdots\sigma_{s-h},T)$. 

It is straightforward to check that we may recover the $231$-avoiding linear extension $\pi=$ \\$\pi_1\pi_2\cdots\pi_{2s-h}$ from the triple \[(\sigma_1\sigma_2\cdots\sigma_m,\sigma_{m+1}\sigma_{m+2}\cdots\sigma_{s-h},T)\in X_s(k,m)\times Y_{s,h}(m)\times Z_s(k,m)\] by reversing the above procedure. Therefore, 
\begin{equation}\label{Eq1}
\sum_{s\geq h+1}A_{231}(U_{\text{spine}=s,2s-h}^\alpha)x^s=\sum_{s\geq h+1}\sum_{k\geq 1}\sum_{m\geq 1}\vert X_s(k,m)\vert\cdot\vert Y_{s,h}(m)\vert\cdot\vert Z_s(k,m)\vert x^s.
\end{equation}
From the definitions of $Y_{s,h}(m)$ and $Z_s(k,m)$ (and the fact that $231$-avoiding permutations are counted by the Catalan numbers), we see that $\vert Y_{s,h}(m)\vert=C_{s-m-h}$ and $\vert Z_s(k,m)\vert={s-m+k\choose s-m}$. Moreover, $\vert X_s(k,m)\vert=\vert X_0(k,m)\vert$, where $X_0(k,m)$ is the set defined in Lemma \ref{Lem1}. When $s<m+h$, $\vert Y_{s,h}(m)\vert=0$. Hence, we may rewrite \eqref{Eq1} as 
\[
\sum_{s\geq h+1}A_{231}(U_{\text{spine}=s,2s-h}^\alpha)x^s=\sum_{k\geq 1}\sum_{m\geq 1}\sum_{s\geq m+h}C_{s-m-h}{s-m+k\choose s-m}\vert X_0(k,m)\vert\cdot x^s.\] Substituting $n=s-m$, invoking Lemma \ref{Lem1}, and rearranging yields 
\[
\sum_{s\geq h+1}A_{231}(U_{\text{spine}=s,2s-h}^\alpha)x^s=\sum_{n\geq h}\sum_{k\geq 1}C_{n-h}{n+k\choose n}\cdot x^n\sum_{m\geq 1}\vert X_0(k,m)\vert x^m\] 
\begin{equation}\label{Eq2}
=\sum_{n\geq h}\sum_{k\geq 1}C_{n-h}{n+k\choose n}x^nx^kC(x)^{k-1}=\frac{1}{C(x)}\sum_{n\geq h}C_{n-h}x^{n}\sum_{k\geq 1}{n+k\choose n}(x\cdot C(x))^k
\end{equation} We now combine the identities $\displaystyle{\sum_{j\geq 0}{\ell+j\choose\ell}z^{j}}=\frac{1}{(1-z)^{\ell+1}}$ and $C(x)=\dfrac{1}{1-x\cdot C(x)}$ to see that \[\sum_{k\geq 1}{n+k\choose n}(x\cdot C(x))^k=\sum_{k\geq 0}{n+k\choose n}(x\cdot C(x))^{k}-1=\frac{1}{(1-x\cdot C(x))^{n+1}}-1=C(x)^{n+1}-1.\] Substituting this into \eqref{Eq2}, we find that \[\sum_{s\geq h+1}A_{231}(U_{\text{spine}=s,2s-h}^\alpha)x^s=\sum_{n\geq h}C_{n-h}(x\cdot C(x))^n-\frac{1}{C(x)}\sum_{n\geq h}C_{n-h}x^n\] \[=(x\cdot C(x))^hC(x\cdot C(x))-x^h.\] The first statement in the theorem now follows from the observation that $A_{231}(U_{\text{spine}=h,h}^\alpha)=1$ and $A_{231}(U_{\text{spine}=s,2s-h}^\alpha)=0$ for all $s\in\{0,1,\ldots,h-1\}$. 

To prove the second statement in the theorem, consider a $231$-avoiding linear extension $q=q_1q_2\cdots q_n$ of $U_{\text{spine}=s,n}^\alpha$, where $n\geq 2s$. Let $j$ be the smallest positive integer such that $n-j$ precedes $n$ in $q$. Because $q$ respects the ordering on $U_{\text{spine}=s,n}^\alpha$, $n-s$ must precede $n$ in $q$. Thus, $j\leq s$. Each $i\in\{1,2,\ldots,n-j-1\}$ must also precede $n$ in $q$ lest $(n-j)ni$ forms a $231$ pattern in $q$. This means that $q=q'nq''$, where $q'$ is a $231$-avoiding linear extension of $U_{\text{spine}=s,n-j}^\alpha$ and $q''$ is a $231$-avoiding permutation of the elements of $\{n-j+1,n-j+2,\ldots,n-1\}$. On the other hand, if $\tau'$ is any $231$-avoiding linear extension of $U_{\text{spine}=s,n-j}^\alpha$ and $\tau''$ is any $231$-avoiding permutation of $\{n-j+1,n-j+2,\ldots,n-1\}$, then $\tau'n\tau''$ is a $231$-avoiding linear extension of $U_{\text{spine}=s,n}^\alpha$. The number of $231$-avoiding permutations of  $\{n-j+1,n-j+2,\ldots,n-1\}$ is $C_{j-1}$, so the second statement in the theorem follows. 
\end{proof}
\begin{remark}\label{Rem1}
Theorem \ref{Thm1} completes the enumeration of the $231$-avoiding linear extensions of the labeled uneven combs $U_{\text{spine}=s,n}^\alpha$. Indeed, the second statement in the theorem provides a recurrence relation that can be used to compute $A_{231}(U_{\text{spine}=s,n}^\alpha)$ for any $n\geq 2s$ so long as we know the values of $A_{231}(U_{\text{spine}=s,m}^\alpha)$ for $s\leq m<2s$. We may use the first statement in the theorem to compute these initial values.  
\end{remark}

\begin{example}
In the proof of Theorem \ref{Thm1}, we showed how to decompose a $231$-avoiding linear extension $\pi$ of $U_{\text{spine}=s,2s-h}^\alpha$ into a triple \[(\sigma_1\sigma_2\cdots\sigma_m,\sigma_{m+1}\sigma_{m+2}\cdots\sigma_{s-h},T)\in X_s(k,m)\times Y_{s,h}(m)\times Z_s(k,m).\] Suppose $s=5$, $h=1$, and $\pi=123846579$. Then $\sigma=8679=(s+3)(s+1)(s+2)(s+4)$, so $k=2$ and $m=3$. Also, $T=\{j\in \{4,5,6,7\}\colon\pi_j\in\{4,5\}\}=\{5,7\}$. Therefore, $\pi$ decomposes into the triple $(867,9,\{5,7\})$. 

Suppose that we were given the triple $(867,9,\{5,7\})$ without knowing $\pi$. We could recover $\pi$ as follows. We know that $1,2,3,4,5$ appear in increasing order in $\pi$ and that the other entries appear in the order $8,6,7,9$. The entries $7$ and $9$ must appear to the right of $5$ in $\pi$ (otherwise, $675$ would be a $231$-pattern). Because $\pi$ is a linear extension of $U_{\text{spine}=5,9}^\alpha$, $3$ must precede $8$ in $\pi$. Finally, the set $T=\{5,7\}$ tells us that $\pi_5=4$ and $\pi_7=5$. This forces $\pi=123846579$. 
\end{example}

\begin{center}
\underline{\bf 3.2 $A_{312}(K_{s,t}^\alpha)$}
\end{center}

Recall from Theorem \ref{ThmY1} that Yakoubov proved that $A_{312}(K_{s,2}^\alpha)=C_{s+1}-C_s$. The following theorem generalizes this result by enumerating the $312$-avoiding linear extensions of $U_{\text{spine}=s,n}^\alpha$ for all $n,s\geq 1$. 

\begin{theorem}\label{Thm2}
Let $s$ be a positive integer. Put $d_s(0)=1$ and $d_s(k)=0$ for all integers $k<0$. For $k\geq 1$, define $d_s(k)$ by $d_s(k)=\sum_{j=1}^sd_s(j-1)d_s(k-j)$. If $n\geq s$, then \[A_{312}(U_{\text{spine}=s,n}^\alpha)=d_s(n-s+1).\]
\end{theorem}
\begin{proof}
Choose $n\geq s$. In any linear extension of $U_{\text{spine}=s,n}^\alpha$, the entries $1,2,\ldots,s$ must form an increasing subsequence. If $\pi$ is such a linear extension that also avoids the pattern $312$, then the first $s-1$ entries of $\pi$ must be $1,2,\ldots,s-1$. This is because if $x>s-1$ and $x$ precedes $s-1$, then $x(s-1)s$ is a $312$ pattern. Therefore, $\pi=12\cdots(s-1)p$, where $p$ is a $312$-avoiding linear extension of the labeled subposet of $U_{\text{spine}=s,n}^\alpha$ obtained by removing from $U_{\text{spine}=s,n}^\alpha$ the elements whose labels are in $[s-1]$. Denote this subposet by $V_{s,n}^\alpha$. If $q$ is any $312$-avoiding linear extension of $V_{s,n}^\alpha$, then $12\cdots(s-1)q$ is a $312$-avoiding linear extension of $U_{\text{spine}=s,n}^\alpha$. Consequently, $A_{312}(U_{\text{spine}=s,n}^\alpha)=A_{312}(V_{s,n}^\alpha)$. 

Say a permutation $\sigma=\sigma_1\sigma_2\cdots\sigma_m$ of some set of integers is $s$-\emph{stratified} if $\sigma_i<\sigma_{s+i}$ for all $1\leq i\leq m-s$. The linear extensions of $V_{s,n}^\alpha$ are precisely the $s$-stratified permutations of the set $\{s,s+1,\ldots,n\}$. By decreasing the entries of such a linear extension by $s-1$, we see that $A_{312}(V_{s,n}^\alpha)$ is equal to the number of $312$-avoiding $s$-stratified permutations in $S_{n-s+1}$. We will show that the number of $312$-avoiding $s$-stratified permutations in $S_k$ is $d_s(k)$; this will complete the proof of the theorem. 

Let $\sigma=\sigma_1\sigma_2\cdots\sigma_k$ be a $312$-avoiding $s$-stratified permutation in $S_k$ for some $k\geq 1$. Suppose $\sigma_j=1$. Because $\sigma$ avoids $312$, all the entries to the left of $1$ in $\sigma$ must be smaller than all of the entries to the right of $1$. Therefore, we may write $\sigma=\tau 1\tau'$, where $\tau$ is a $312$-avoiding $s$-stratified permutation of $\{2,3,\ldots,j\}$ and $\tau'$ is a $312$-avoiding $s$-stratified permutation of $\{j+1,j+2,\ldots,k\}$. The entry $s+1$ cannot precede $1$ in $\sigma$ because $\sigma$ is $s$-stratified. This forces $j\leq s$. On the other hand, if $q$ is a $312$-avoiding $s$-stratified permutation of $\{2,3,\ldots,j\}$ and $q'$ is a $312$-avoiding $s$-stratified permutation of $\{j+1,j+2,\ldots,k\}$ for some $j\leq s$, then $q1q'$ is a $312$-avoiding $s$-stratified permutation in $S_k$. It follows that the number of $312$-avoiding $s$-stratified permutations in $S_k$ is $d_s(k)$.   
\end{proof}

\begin{center}
\underline{\bf 3.3 $A_{321}(K_{s,t}^\alpha)$}
\end{center}

In this section, we write $\{a_1,a_2,\ldots,a_k\}_<$ to denote a set of integers $a_1,a_2,\ldots,a_k$ such that $a_1<a_2<\cdots<a_k$. In other words, the subscript $``<"$ is simply used to emphasize that we have written the elements of the set in increasing order. 

We now consider $321$-avoiding linear extensions of $K_{s,t}^\alpha$. Recall that an entry $\pi_i$ of a permutation $\pi=\pi_1\pi_2\cdots\pi_n\in S_n$ is called a \emph{right-to-left minimum} if $\pi_i<\pi_j$ for all $j\in\{i+1,i+2,\ldots,n\}$. The right-to-left minima of $\pi$ necessarily appear in increasing order. If $\pi$ avoids the pattern $321$, then the entries that are not right-to-left minima also appear in increasing order. Indeed, suppose $\pi_i>\pi_j$ and $i<j$ for some entries $\pi_i,\pi_j$ that are not right-to-left minima. Since $\pi_j$ is not a right-to-left minimum, there is some $k>j$ such that $\pi_k<\pi_j$. Then $\pi_i\pi_j\pi_k$ forms a $321$ pattern in $\pi$, contrary to assumption. It follows that each $321$-avoiding permutation in $S_n$ is uniquely determined by specifying the entries that are not right-to-left minima and the positions in the permutation that those entries occupy.   

Assume $s,t\geq 2$. In any linear extension of $K_{s,t}^\alpha$, the entries $1,2,\ldots,s$ must be right-to-left minima because they appear in increasing order. Let us fix a subset $W=\{w_1,w_2,\ldots,w_k\}_<$ of $\{s+1,s+2,\ldots,st\}$. Define $\Omega_{s,t}(W)$ to be the collection of all sets of integers $I=\{i_1,i_2,\ldots,i_k\}_<$ satisfying $w_\ell-s+\ell\leq i_\ell\leq w_\ell-1$ for all $\ell$. Let $\Lambda_{s,t}(W)$ be the set of all $321$-avoiding linear extensions of $K_{s,t}^\alpha$ whose right-to-left minima are precisely the elements of $[st]\setminus W$. Our goal is to construct an injection from $\Omega_{s,t}(W)$ to $\Lambda_{s,t}(W)$ which is a bijection if $t=2$. By estimating the size of $\Omega_{s,2}(W)$, we will then be able to deduce a lower bound for $A_{321}(K_{s,2}^\alpha)$. From there, we will deduce lower bounds for $A_{321}(K_{s,t}^\alpha)$ for general $t$.  

Suppose we are given a set of integers $I=\{i_1,i_2,\ldots,i_k\}_<$ in $\Omega_{s,t}(W)$. Define a permutation $\mu=\mu_W(I)=\mu_1\mu_2\cdots\mu_{st}$ as follows. For each $\ell\in[k]$, let $\mu_{i_\ell}=w_\ell$. Place the elements of $[st]\setminus W$ in increasing order in the remaining unfilled positions of $\mu$. For example, if $s=4$, $t=2$, $W=\{6,7\}$, and $I=\{4,6\}$, then $\mu_W(I)=12364758$. In this example, $\mu_W(I)$ is an element of $\Lambda_{s,t}(W)$. The following lemma tells us that this is not a coincidence. 

\begin{lemma}\label{Lem2}
Fix integers $s,t\geq 2$ and a subset $W=\{w_1,w_2,\ldots,w_k\}_<$ of $\{s+1,s+2,\ldots,st\}$. Preserving the notation from the previous two paragraphs, let $I=\{i_1,i_2,\ldots,i_k\}_<$ be a set of integers in $\Omega_{s,t}(W)$. The permutation $\mu_W(I)$ is an element of $\Lambda_{s,t}(W)$.  
\end{lemma}
\begin{proof}
The permutation $\mu_W(I)$ is a union of two increasing subsequences (the elements of $W$ form one, and the elements of $[st]\setminus W$ form the other), so $\mu_W(I)$ avoids $321$. Choose some $\ell\in[k]$. By the definition of $\mu_W(I)$, the $i_\ell^\text{th}$ entry of $\mu_W(I)$ is $w_\ell$. Because $i_\ell\leq w_\ell-1$, there is some entry $a<w_\ell$ that appears to the right of $w_\ell$ in $\mu_W(I)$. This means that $w_\ell$ is not a right-to-left minimum. Since $\ell$ was arbitrary, none of the elements of $W$ are right-to-left minima of $\mu_W(I)$. 

Suppose there is some $z\in[st]\setminus W$ that is not a right-to-left minimum of $\mu_W(I)$. This means that there is some entry $z'<z$ that appears to the right of $z$ in $\mu_W(I)$. Since $\mu_W(I)$ avoids $321$, $z'$ must be a right-to-left minimum of $\mu_W(I)$. It follows that $z'\in[st]\setminus W$. However, since $z\in [st]\setminus W$, this contradicts the fact that the elements of $[st]\setminus W$ form an increasing subsequence of $\mu_W(I)$. From this contradiction, we deduce that the right-to-left minima of $\mu_W(I)$ are precisely the elements of $[st]\setminus W$. 

To see why $\mu_W(I)$ is a linear extension of  $K_{s,t}^\alpha$, note first that the entries $1,2,\ldots,s$ appear in increasing order. Suppose, by way of contradiction, that there is some $y\in[st-s]$ such that $y+s$ appears to the left of $y$ in $\mu_W(I)$. Then $y+s$ is not a right-to-left minimum of $\mu_W(I)$. In addition, $y$ is a right-to-left minimum of $\mu_W(I)$ because $\mu_W(I)$ avoids $321$. It follows that $y+s\in W$ and $y\in[st]\setminus W$. Say $y+s=w_\ell$. By the definition of $\mu_W(I)$, $w_\ell$ is the $i_\ell^{\text{th}}$ entry of $\mu_W(I)$. By hypothesis, $i_\ell\geq w_\ell-s+\ell=y+\ell$.  Since the elements of $W$ form an increasing subsequence of $\mu_W(I)$, there are exactly $\ell-1$ elements of $W$ to the left of $w_\ell$. Consequently, there are at least $y$ elements of $[st]\setminus W$ to the left of $w_\ell$. The elements of $[st]\setminus W$ also appear in increasing order in $\mu_W(I)$, so the smallest $y$ elements of $[st]\setminus W$ appear to the left of $w_\ell$. However, this is a contradiction since $y$ is necessarily one of the $y$ smallest elements of $[st]\setminus W$. Hence, $y$ appears before $y+s$ in $\mu_W(I)$ for each $y\in[st-s]$. This proves that $\mu_W(I)$ is a linear extension of $K_{s,t}^\alpha$. 
\end{proof}

Lemma \ref{Lem2} tells us that we have a well-defined map $\mu_W\colon\Omega_{s,t}(W)\to\Lambda_{s,t}(W)$. This map, which we defined in the paragraph immediately preceding Lemma \ref{Lem2}, is actually an injection. Indeed, if we are given $\mu_W(I)$, we can easily recover the set $I$ from the observation that $i\in I$ if and only if the $i^\text{th}$ entry of $\mu_W(I)$ is in $W$. This proves the first statement in the following theorem.   

\begin{theorem}\label{Thm3}
Preserve the notation from the preceding paragraph. The map $\mu_W\colon\Omega_{s,t}(W)\to\Lambda_{s,t}(W)$ is a well-defined injection. If $t=2$, this map is bijective. 
\end{theorem}
\begin{proof}
We are left to show that $\mu_W\colon\Omega_{s,t}(W)\to\Lambda_{s,t}(W)$ is surjective if $t=2$. We may assume $W$ is nonempty. Indeed, if $W=\emptyset$, then $\mu_W$ is a bijection because $\Omega_{s,t}(W)=\{\emptyset\}$ and $\Lambda_{s,t}(W)=\{12\cdots(st)\}$. Suppose $t=2$, and choose some $\pi=\pi_1\pi_2\cdots\pi_{2s}\in\Lambda_{s,2}(W)$. For each $\ell\in[k]$, let $i_\ell$ be the integer satisfying $\pi_{i_\ell}=w_\ell$. Put $I=\{i_1,i_2,\ldots,i_k\}$. We will prove that $\pi=\mu_W(I)$. To do so, we simply need to show that $I\in\Omega_{s,2}(W)$. Indeed, we know the sets $W$ and $[2s]\setminus W$ form increasing subsequences of $\pi$ because $[2s]\setminus W$ consists of the right-to-left minima of $\pi$ and $\pi$ avoids $321$. Choose some $\ell\in[k]$. We need to show that $w_\ell-s+\ell\leq i_\ell\leq w_\ell-1$. 

Note that $w_\ell\in\{s+1,s+2,\ldots,2s\}$ by the definition of $W$. Because $\pi$ is a linear extension of $K_{s,2}^\alpha$, $w_\ell-s$ precedes $w_\ell$ in $\pi$. The entries $1,2,\ldots,w_\ell-s$ all precede $w_\ell$ in $\pi$ because they appear in increasing order. The entries $w_1,w_2,\ldots,w_{\ell-1}$ also precede $w_\ell$ in $\pi$ since the elements of $W$ appear in increasing order. Observe that the sets $\{w_1,w_2,\ldots,w_{\ell-1}\}$ and $\{1,2,\ldots,w_\ell-s\}$ are disjoint because all of the elements of the latter set are less than or equal to $s$. This proves that there are at least $w_\ell-s+\ell-1$ entries to the left of $w_\ell$ in $\pi$, which means that $i_\ell\geq w_\ell-s+\ell$. 
 
Now, because $w_\ell$ is not a right-to-left minimum of $\pi$, there is an entry $x<w_\ell$ that appears to the right of $w_\ell$. Since $\pi$ avoids $321$, all of the entries to the left of $w_\ell$ are elements of the set $\{1,2,\ldots,w_\ell-1\}\setminus\{x\}$. This shows that there are at most $w_\ell-2$ entries to the left of $w_\ell$, so $i_\ell\leq w_\ell-1$.   
\end{proof}
 
Using the second statement in Theorem \ref{Thm3}, we have computed $A_{321}(K_{s,2}^\alpha)$ for $1\leq s\leq 13$. These values, starting with $s=1$, are \[1, 3, 13, 67, 378, 2244, 13737, 85767, 542685, 3466515, 22298796, 144210388, 936575968.\] This extends the list of known values of this sequence since Yakoubov only computed $A_{321}(K_{s,2}^\alpha)$ for $1\leq s\leq 6$. 

For fixed $t\geq 2$, it is natural to ask about the growth rate $\displaystyle{\lim_{s\to\infty}\sqrt[s]{A_{321}(K_{s,t}^\alpha)}}$ of the sequence $(A_{321}(K_{s,t}^\alpha))_{s=1}^\infty$. First of all, how do we know this limit even exists? We say a sequence of real numbers $(a_m)_{m=1}^\infty$ is supermultiplicative if $a_ma_n\leq a_{m+n}$ for all positive integers $m,n$. The multiplicative version of Fekete's lemma  \cite{Fekete} states that if $(a_m)_{m=1}^\infty$ is a supermultiplicative sequence, then $\displaystyle{\lim_{m\to\infty}\sqrt[m]{a_m}}$ exists and equals $\displaystyle{\sup_{m\geq 1}\sqrt[m]{a_m}}$. We claim that $(A_{321}(K_{s,t}^\alpha))_{s=1}^\infty$ is a supermultiplicative sequence. 

If $\pi=\pi_1\pi_2\cdots\pi_m\in S_m$ and $\tau=\tau_1\tau_2\cdots\tau_n\in S_n$ are permutations, then $\pi\oplus\tau$ is the permutation $p_1p_2\cdots p_{m+n}\in S_{m+n}$ defined by \[p_i=\begin{cases} \pi_i, & \mbox{if } 1\leq i\leq m; \\ m+\tau_{i-m}, & \mbox{if } m+1\leq i\leq m+n. \end{cases}\] It is easy to see that if $\pi$ and $\tau$ are $321$-avoiding linear extensions of $K_{s,t}^\alpha$ and $K_{s',t}^\alpha$, respectively, then $\pi\oplus\tau$ is a $321$-avoiding linear extension of $K_{s+s',t}^\alpha$. Therefore, $A_{321}(K_{s,t}^\alpha)A_{321}(K_{s',t}^\alpha)\leq A_{321}(K_{s+s',t}^\alpha)$. By Fekete's lemma, the limit $\displaystyle{\lim_{s\to\infty}\sqrt[s]{A_{321}(K_{s,t}^\alpha)}}$ exists. 

Theorem \ref{Thm4} provides upper and lower bounds for the growth rate $\displaystyle{\lim_{s\to\infty}\sqrt[s]{A_{321}(K_{s,t}^\alpha)}}$ for each fixed $t\geq 2$. The proof of the lower bound relies heavily upon Theorem \ref{Thm3}, which is why we cared so deeply about proving that result. The idea behind the proof is to find a lower bound for the size of the set $\Omega_{s,2}(W)$ and then use the bijection from Theorem \ref{Thm3} to deduce a lower bound for $\Lambda_{s,2}(W)$. We derive a lower bound for $\vert\Omega_{s,t}(W)\vert$ in the following lemma. 

\begin{lemma}\label{Lem3}
Let $W=\{w_1,w_2,\ldots,w_k\}_<$ be a subset of  $\{s+1,s+2,\ldots,st\}$, and let $\Omega_{s,t}(W)$ be as in Theorem \ref{Thm3}. We have \[\vert \Omega_{s,t}(W)\vert\geq \frac{s-k}{s}{s+k-1\choose k}.\]
\end{lemma}    
\begin{proof}
Let us begin by choosing a lattice path $L$ from $(0,0)$ to $(s-1,k)$ that uses steps $(0,1)$ and $(1,0)$ and never passes above the line $y=x$. Let $\mathcal J(L)$ be the set of all $j\in[s+k-1]$ such that the $j^\text{th}$ step in $L$ is a $(0,1)$ step. Let us write $\mathcal J(L)=\{j_1,j_2,\ldots,j_k\}_<$. Observe that $2\ell\leq j_\ell\leq s+\ell-1$. Let $\mathcal I(\mathcal J(L))=\{i_1,i_2,\ldots,i_k\}$, where $i_\ell=j_\ell+w_\ell-s-\ell$. For all $\ell\in[k-1]$, we have \[i_{\ell+1}-i_\ell=j_{\ell+1}-j_\ell+w_{\ell+1}-w_\ell-1\geq j_{\ell+1}-j_\ell\geq 1,\] so $i_1<i_2<\cdots<i_k$. From the inequalities $2\ell\leq j_\ell\leq s+\ell-1$, we obtain $w_\ell-s+\ell\leq i_\ell\leq w_\ell-1$. This shows that $\mathcal I(\mathcal J(L))\in\Omega_{s,t}(W)$. The maps $L\mapsto\mathcal J(L)$ and $\mathcal J(L)\mapsto\mathcal I(\mathcal J(L))$ are both injective, so $\vert \Omega_{s,t}(W)\vert$ is at least the number of possible choices for $L$. A variant of Bertrand's ballot theorem \cite{Bertrand} states that the number of choices for $L$ is \[\frac{s-k}{s}{s+k-1\choose k}. \qedhere\]
\end{proof}

\begin{theorem}\label{Thm4}
For $t\geq 2$, \[4^{t-2}(3+2\sqrt 2)\leq\lim_{s\to\infty}\sqrt[s]{A_{321}(K_{s,t}^\alpha)}\leq\frac{(2t-1)^{2t-1}}{(t-1)^{t-1}t^t}.\] 
\end{theorem}
\begin{proof}
We first prove the lower bound when $t=2$. Recall that $\Lambda_{s,2}(W)$ is the set of $321$-avoiding linear extensions of $K_{s,2}^\alpha$ whose right-to-left minima are the elements of $[2s]\setminus W$. This means that 
\begin{equation}\label{Eq3} 
A_{321}(K_{s,2}^\alpha)=\sum_{W\subseteq\{s+1,\ldots,2s\}}\vert\Lambda_{s,2}(W)\vert.
\end{equation} From Theorem \ref{Thm3} and Lemma \ref{Lem3}, we deduce that if $W$ is a $k$-element subset of \\$\{s+1,s+2,\ldots,2s\}$, then $\displaystyle{\vert\Lambda_{s,2}(W)\vert\geq\frac{s-k}{s}{s+k-1\choose k}}$. Partitioning the subsets of \\$\{s+1,s+2,\ldots,2s\}$ in the summation in \eqref{Eq3} according to their sizes, we find that \[A_{321}(K_{s,2}^\alpha)=\sum_{k=0}^{s}\sum_{\substack{W\subseteq\{s+1,\ldots,2s\}\\ \vert W\vert=k}}\vert\Lambda_{s,2}(W)\vert\geq\sum_{k=0}^{s}\frac{s-k}{s}{s\choose k}{s+k-1\choose k}\] \[\geq \frac{s-\left\lfloor s/\sqrt 2\right\rfloor}{s}{s\choose \left\lfloor s/\sqrt 2\right\rfloor}{s+\left\lfloor s/\sqrt 2\right\rfloor-1\choose \left\lfloor s/\sqrt 2\right\rfloor}.\] We chose to extract the term in which $k=\left\lfloor s/\sqrt 2\right\rfloor$ in order to maximize our obtained lower bound. Using Stirling's approximation, it is straightforward to show that \[\lim_{s\to\infty}\sqrt[s]{\frac{s-\left\lfloor s/\sqrt 2\right\rfloor}{s}{s\choose \left\lfloor s/\sqrt 2\right\rfloor}{s+\left\lfloor s/\sqrt 2\right\rfloor-1\choose \left\lfloor s/\sqrt 2\right\rfloor}}=\frac{(1+1/\sqrt 2)^{1+1/\sqrt 2}}{\left((1/\sqrt 2)^{1/\sqrt 2}\right)^2(1-1/\sqrt 2)^{1-1/\sqrt 2}}=3+2\sqrt 2.\] 

Next, suppose $t\geq 3$. Let $\tau$ be a $321$-avoiding linear extension of $K_{s,2}^\alpha$. For each $r\in\{2,3,\ldots,$ $t-1\}$, let $\sigma_r$ be a $321$-avoiding permutation of the elements of $\{rs+1,rs+2,\ldots,(r+1)s\}$. The permutation $\tau\sigma_2\sigma_3\cdots\sigma_{t-1}$ is a $321$-avoiding linear extension of $K_{s,t}^\alpha$. There are $A_{321}(K_{s,2}^\alpha)C_s^{t-2}$ ways to choose the permutations $\tau,\sigma_2,\sigma_3,\ldots,\sigma_{t-1}$, so \[\lim_{s\to\infty}\sqrt[s]{A_{321}(K_{s,t}^\alpha)}\geq\lim_{s\to\infty}\sqrt[s]{A_{321}(K_{s,2}^\alpha)C_s^{t-2}}\geq 4^{t-2}(3+2\sqrt 2).\]

Let us now proceed to prove the upper bound stated in the theorem. Suppose we wish to construct a $321$-avoiding linear extension of $K_{s,t}^\alpha$ that has $st-k$ right-to-left minima. We must have $k\leq st-s$ since the entries $1,2,\ldots,s$ will be right-to-left minima. By the same token, all entries that are not right-to-left minima must belong to the set $\{s+1,s+2,\ldots,st\}$. There are ${st-s\choose k}$ ways to choose the $k$ entries that are not right-to-left minima. There are then at most ${st\choose k}$ ways to choose the positions in the permutation that these $k$ entries will occupy. After we make these choices, the permutation is determined by the assumption that it avoids the pattern $321$. Indeed, the right-to-left minima must appear in increasing order, and the other entries must also appear in increasing order. 

This shows that 
\begin{equation}\label{Eq4}
A_{321}(K_{s,t}^\alpha)\leq\sum_{k=0}^{st-s}{(t-1)s\choose k}{st\choose k}.
\end{equation}
Let $u_{s,t}(k)={(t-1)s\choose k}{st\choose k}$. Let $r_s(t)$ be one of the integers $k\in\{0,1,\ldots,st-s\}$ that maximizes $u_{s,t}(k)$, and put $\gamma_s(t)=r_s(t)/s$. We have \[\frac{u_{s,t}(k)}{u_{s,t}(k+1)}=\frac{(k+1)^2}{((t-1)s-k)(st-k)},\] so $u_{s,t}(k)\leq u_{s,t}(k+1)$ if and only if $(k+1)^2\leq ((t-1)s-k)(st-k)$. This occurs if and only if $k\leq\dfrac{t(t-1)s^2-1}{(2t-1)s+2}$. Consequently, 
\begin{equation}\label{Eq5}
\gamma_s(t)=\frac 1s\left(\dfrac{t(t-1)s^2-1}{(2t-1)s+2}+O(1)\right)=\dfrac{t(t-1)}{2t-1}+O(1/s)
\end{equation} for each fixed $t\geq 2$. Using \eqref{Eq4}, we find that \[\lim_{s\to\infty}\sqrt[s]{A_{321}(K_{s,t}^\alpha)}\leq\lim_{s\to\infty}\sqrt[s]{(st-s+1)u_{s,t}(\gamma_s(t)s)}.\] To ease notation, let us write $f(x)=x^x$ and $\gamma(t)=\dfrac{t(t-1)}{2t-1}$. Using Stirling's approximation and \eqref{Eq5}, it is straightforward to show that 
\begin{equation}\label{Eq6}
\lim_{s\to\infty}\sqrt[s]{(st-s+1)u_{s,t}(\gamma_s(t)s)}=\frac{f(t-1)f(t)}{f(\gamma(t))^2f(t-1-\gamma(t))f(t-\gamma(t))}.
\end{equation} Elementary algebraic manipulations show that the right-hand side of \eqref{Eq6} is equal to $\dfrac{(2t-1)^{2t-1}}{t^t(t-1)^{t-1}}$. 
\end{proof}

\begin{center}
\underline{\bf 3.4 $A_{321}(K_{s,t}^\beta)$}
\end{center}
Until now, we have limited our focus to linear extensions of $\alpha$-labeled combs. At this point, we turn our attention to $\beta$-labeled combs. Recall from the end of the introduction that the $\beta$-labeled comb $K_{s,t}^\beta$ is obtained from $K_{s,t}$ by assigning each element $e_{i,j}$ the label $(i-1)t+j$. 

In her paper, Yakoubov finds explicit formulas for $A_{\tau}(K_{s,2}^\beta)$ for every length-$3$ pattern $\tau$. She also finds formulas for $A_{\tau}(K_{s,t}^\beta)$ for all $t\geq 3$ whenever $\tau$ is a length-$3$ pattern other than $321$. It is our goal to give a general formula for $A_{321}(K_{s,t}^\beta)$. We will also show that for each fixed $s\geq 2$, $\displaystyle{\lim_{t\to\infty}\sqrt[t]{A_{321}(K_{s,t}^\beta)}=2^s}$. 

\begin{theorem}\label{Thm5}
For $t\geq 2$, let \[F_{2,t}(k)=\begin{cases} 1, & \mbox{if } 2\leq k\leq t+1; \\ 0, & \mbox{otherwise.} \end{cases}\]
For $s\geq 3$ and $t\geq 2$, define $F_{s,t}(k)$ recursively by \[F_{s,t}(k)=\begin{cases} \displaystyle{\sum_{i=s-1}^{k-1}F_{s-1,t}(i)\sum_{j=\max\{0,k-(s-2)t-2\}}^{t-1}{k-i-1\choose j}}, & \mbox{if } s\leq k\leq(s-1)t+1; \\ 0, & \mbox{otherwise.} \end{cases}\] We have \[A_{321}(K_{s,t}^\beta)=\sum_{k=s}^{(s-1)t+1}{st-k\choose t-1}F_{s,t}(k).\] 
\end{theorem}
\begin{proof}
Let $\mathcal F_{s,t}(k)$ be the set of words $w=w_1w_2\cdots w_k$ of length $k$ over the alphabet $\{a,b,c\}$ that satisfy the following properties: 
\begin{enumerate}[(i)]
\item There are indices $\ell_1,\ell_2\ldots,\ell_s\in[k]$ such that $1=\ell_1<\ell_2<\cdots<\ell_s=k$ and $w_{\ell_1}=w_{\ell_2}=\cdots=w_{\ell_s}=c$; 
\item If $j\in[k]\setminus\{\ell_1,\ell_2,\ldots,\ell_s\}$, then $w_j\neq c$;
\item For each $i\in[s]$, $\ell_i\leq (i-1)t+1$. 
\item If $B_i$ denotes the number of occurrences of the letter $b$ in the word $w_{\ell_{i-1}+1}w_{\ell_{i-1}+2}\cdots w_{\ell_i-1}$, then $\ell_i-(i-2)t-2\leq B_i\leq t-1$.
\end{enumerate}
We claim that $\vert\mathcal F_{s,t}(k)\vert=F_{s,t}(k)$. To see this, first observe that (i) forces $\mathcal F_{s,t}(k)=\emptyset$ if $k<s$. Similarly, (i) and (iii) guarantee that $\mathcal F_{s,t}(k)=\emptyset$ if $k>(s-1)t+1$. Now, suppose $s\leq k\leq (s-1)t+1$. If $s=2$, one may use (i), (iii), and (iv) to see that $\mathcal F_{s,t}(k)=\{cb^{k-2}c\}$. Hence, $\vert\mathcal F_{2,t}(k)\vert=F_{2,t}(k)$. 

Assume $s\geq 3$. We will build a word $w=w_1w_2\cdots w_k$ in $\mathcal F_{s,t}(k)$. Preserve the notation from properties (i)--(iv) above. The word $w_1w_2\cdots w_{\ell_{s-1}}$ must be an element of $\mathcal F_{s-1,t}(\ell_{s-1})$. The word $w_{\ell_{s-1}+1}w_{\ell_{s-1}+2}\cdots w_{k-1}$ can use only the letters $a$ and $b$, and $w_k$ must be $c$. By (iv), there are $\displaystyle{\sum_{j=\max\{0,k-(s-2)t-2\}}^{t-1}{k-\ell_{s-1}-1\choose j}}$ ways to choose which of the letters in $w_{\ell_{s-1}+1}w_{\ell_{s-1}+2}\cdots w_{k-1}$ are $b$'s. There are $\vert \mathcal F_{s-1,t}(\ell_{s-1})\vert$ ways to choose $w_1w_2\cdots w_{\ell_{s-1}}$. It follows that 
\[\vert\mathcal F_{s,t}(k)\vert=\sum_{\ell_{s-1}=s-1}^{k-1}\vert\mathcal F_{s-1,t}(\ell_{s-1})\vert\sum_{j=\max\{0,k-(s-2)t-2\}}^{t-1}{k-\ell_{s-1}-1\choose j}.\] Our claim follows by induction on $s$. 

Now, let $\pi=\pi_1\pi_2\cdots\pi_{st}$ be a $321$-avoiding linear extension of $K_{s,t}^\beta$. Let $k$ be the index such that $\pi_k=(s-1)t+1$. The elements $1,t+1,2t+1,\ldots,(s-1)t+1$ must form an increasing subsequence of $\pi$, as must $(s-1)t+1,(s-1)t+2,\ldots,st$. This forces $s\leq k\leq (s-1)t+1$. We can encode the string $\pi_1\pi_2\cdots\pi_k$ with a word $w=w_1w_2\cdots w_k$ as follows. If $\pi_\ell\equiv 1\pmod t$, let $w_\ell=c$. If $\pi_\ell$ is a left-to-right maximum of $\pi$ (meaning $\pi_\ell$ is greater than all entries to its left) and $\pi_\ell\not\equiv 1\pmod t$, let $w_\ell=b$. Otherwise, let $w_\ell=a$. We will show that $w\in\mathcal F_{s,t}(k)$ and that every word in $\mathcal F_{s,t}(k)$ is obtained uniquely in this way.  

For each $i\in[s]$, let $\ell_i$ be the index such that $\pi_{\ell_i}=(i-1)t+1$. Properties (i), (ii), and (iii) follow immediately from the definition of $w$ and the fact that $\pi$ is a linear extension of $K_{s,t}^\beta$. Furthermore, $B_i$ is equal to the number of elements of $\{(i-2)t+2,(i-2)t+3,\ldots,(i-1)t\}$ lying between $\pi_{\ell_{i-1}}$ and $\pi_{\ell_i}$ in $\pi$. This implies that $B_i\leq t-1$. All of the entries of $\pi$ to the left of $\pi_{\ell_i}=(i-1)t+1$ are less than $(i-1)t+1$. At most $(i-2)t+1$ of these entries are in the set $\{1,2,\ldots,(i-2)t+1\}$, and exactly $B_i$ are in the set $\{(i-2)t+2,(i-2)t+3,\ldots,(i-1)t\}$. Therefore, $\ell_i-1\leq (i-2)t+1+B_i$. It follows that (iv) holds, so $w\in\mathcal F_{s,t}(k)$.  

Suppose, now, that we are given the word $w\in\mathcal F_{s,t}(k)$. We may recover $\pi_1\pi_2\cdots\pi_k$ as follows. First, we know that $\pi_1=1$. Suppose we have already determined $\pi_1,\pi_2,\ldots,\pi_\ell$. Let $j$ be the largest integer such that $(j-1)t+1\in\{\pi_1,\pi_2,\ldots,\pi_\ell\}$. If $w_{\ell+1}=c$, then we must put $\pi_{\ell+1}=jt+1$. If $w_{\ell+1}=b$, then $\pi_{\ell+1}$ must be the smallest element of $\{(j-1)t+2,(j-1)t+3,\ldots,jt\}$ that is not in the set $\{\pi_1,\pi_2,\ldots,\pi_\ell\}$. The inequality $B_{j+1}\leq t-1$ in (iv) guarantees that such an element exists. If $w_{\ell+1}=a$, then $\pi_{\ell+1}$ must be an element of $\{1,2,\ldots,(j-1)t\}$ that is not in the set $\{\pi_1,\pi_2,\ldots,\pi_\ell\}$. One may use the inequality $\ell_{j+1}-(j-1)t-2\leq B_{j+1}$ in (iv) to show that an element of $\{1,2,\ldots,(j-1)t\}\setminus\{\pi_1,\pi_2,\ldots,\pi_\ell\}$ actually exists. In fact, $\pi_{\ell+1}$ must be the smallest such element. Indeed, if there were some $x\in\{1,2,\ldots,(j-1)t\}\setminus\{\pi_1,\pi_2,\ldots,\pi_\ell\}$ with $x<\pi_{\ell+1}$, then the entries $(j-1)t+1, \pi_{\ell+1}$, and $x$ would form a $321$ pattern in $\pi$. This shows that each word in $\mathcal F_{s,t}(k)$ is obtained uniquely from an initial string $\pi_1\pi_2\cdots\pi_k$ of a $321$-avoiding linear extension of $K_{s,t}^\beta$. 

We now construct a $321$-avoiding linear extension $\pi=\pi_1\pi_2\cdots\pi_{st}$ of $K_{s,t}^\beta$. Let $k$ be the index such that $\pi_k=(s-1)t+1$. By the preceding discussion, there are $F_{s,t}(k)$ ways to choose $\pi_1\pi_2\cdots\pi_k$. After doing so, there are $st-k$ entries left to use to form the string $\pi_{k+1}\pi_{k+2}\cdots\pi_{st}$. Those entries of $\pi_{k+1}\pi_{k+2}\cdots\pi_{st}$ that are less than $(s-1)t+1$ should appear in increasing order since $\pi$ must avoid $321$. The entries $(s-1)t+2,(s-1)t+3,\ldots,st$ should form an increasing subsequence of $\pi_{k+1}\pi_{k+2}\cdots\pi_{st}$ because $\pi$ is a linear extension of $K_{s,t}^\beta$. Therefore, $\pi_{k+1}\pi_{k+2}\cdots\pi_{st}$ is completely determined by the choice of the $t-1$ positions in this string that the entries $(s-1)t+2,(s-1)t+3,\ldots,st$ will occupy. There are ${st-k\choose t-1}$ ways to make this choice. Summing over $k$, we obtain the formula in the final statement of the theorem. 
\end{proof}
\begin{remark}\label{Rem2}
In the preceding proof, we showed that $F_{s,t}(k)=\vert\mathcal F_{s,t}(k)\vert$. Each element of $\mathcal F_{s,t}(k)$ is a word of length $k$ over the alphabet $\{a,b,c\}$ in which the letter $c$ appears exactly $s$ times. The number of such words is $2^{k-s}{k\choose s}$, so it follows from Theorem \ref{Thm5} that \[A_{321}(K_{s,t}^\beta)\leq\sum_{k=s}^{(s-1)t+1}2^{k-s}{k\choose s}{st-k\choose t-1}.\]
\end{remark}

We have used Theorem \ref{Thm5} to calculate several new values of $A_{321}(K_{s,t}^\beta)$ for various $s$ and $t$. Some of these values are shown in Table \ref{Tab1}. We have entered the values of $A_{321}(K_{3,t}^\beta)$ for $t\leq 40$ and the values of $A_{321}(K_{4,t}^\beta)$ for $t\leq 20$ into the Online Encyclopedia of Integer Sequences (sequences A275941 and A275942).

\begin{minipage}{\textwidth}  
\begin{center}
{\small

     \begin{tabular}{ | l |p{1cm}|p{1.67cm}|p{1.7cm}|p{2.18cm}|p{2.2cm}|p{2.55cm}|}
    \hline
    $s$ & $t=2$ & $t=3$ & $t=4$ & $t=5$ & $t=6$ & $t=7$ \\ \hline
    2 &3&10&35&126&462&1716 \\ 
    3 &12 &127 &1222 &11096 &97140 &830152 \\ 4&55&1866&49523&1147175&24446239&
    492996938\\ 5&273&29839&2182844&128783730&6664055770&
    316066050507\\ 6 &1428&504265&101666026&15268771939&
    1917617336329&213823357879553 \\ 
    7 & 7752&8859742&4922704260&1881489465581 & & \\ 8 & 43263&160216631 & & & & \\ 9 & 246675&2962451668 & & & & \\
    \hline
    \end{tabular}\captionof{table}{Values of $A_{321}(K_{s,t}^\beta)$ for some small $s,t$.}\label{Tab1}
}
\end{center}
\end{minipage}

Suppose $P$ is a finite poset whose Hasse diagram is a rooted tree. For $x\in P$, let $d(x)$ be the number of elements $y\in P$ such that $y\geq_P x$. A well-known formula due to Knuth \cite{Knuth} states that the number of linear extensions of $P$ is \[\frac{n!}{\prod_{x\in P}d(x)},\] where $n=\vert P\vert$. It follows that the number of linear extensions of $K_{2,t}^\beta$ is $\frac 12{2t\choose t}$. Observe that each linear extension of $K_{2,t}^\beta$ is a union of two increasing subsequences (one formed by the entries $1,2,\ldots,t$ and the other formed by $t+1,t+2,\ldots,2t$). This means that every linear extension of $K_{s,t}^\beta$ avoids $321$, so $A_{321}(K_{2,t}^\beta)=\frac 12{2t\choose t}$. As a consequence, $\displaystyle{\lim_{t\to\infty}\sqrt[t]{A_{321}(K_{2,t}^\beta)}}=4$. In the following theorem, we use Theorem \ref{Thm5} to prove a more general statement. 

\begin{theorem}\label{Thm6}
For each $s\geq 2$, \[\lim_{t\to\infty}\sqrt[t]{A_{321}(K_{s,t}^\beta)}=2^s.\] 
\end{theorem}
\begin{proof}
We have just seen that this theorem holds when $s=2$, so assume $s\geq 3$. Let $F_{s,t}(k)$ be as defined in the statement of Theorem \ref{Thm5}. We will prove that 
\begin{equation}\label{Eq7}
F_{s,t}(k)\geq\frac{2^{k-s}}{(k-2)(k-3)\cdots(k-s+1)}\hspace{.5cm}\text{whenever}\hspace{.5cm}s\leq k\leq (s-1)t+1. 
\end{equation}

If $3\leq k\leq 2t+1$, then \[F_{3,t}(k)=\sum_{i=2}^{k-1}F_{2,t}(i)\sum_{j=\max\{0,k-t-2\}}^{t-1}{k-i-1\choose j}\geq F_{2,t}(2){k-3\choose \left\lfloor (k-3)/2\right\rfloor}\geq \frac{2^{k-3}}{k-2}\] because $F_{2,t}(2)=1$. This proves \eqref{Eq7} in the case $s=3$. Now, assume $s\geq 4$, $t\geq 2$, and $s\leq k\leq(s-1)t+1$. We have $0<(s-3)t+3-s$, so $k\leq(s-1)t+1< 2(s-2)t+4-s$. This implies that $2k-2(s-2)t-4<k-s$, so 
\begin{equation}\label{Eq10}
s-1<k-1-2\omega_{s,t}(k)\leq k-1,
\end{equation} where $\omega_{s,t}(k)=\max\{0,k-(s-2)t-2\}$. Now, 
\[F_{s,t}(k)=\sum_{i=s-1}^{k-1}F_{s-1,t}(i)\sum_{j=\omega_{s,t}(k)}^{t-1}{k-i-1\choose j}\] 
\begin{equation}\label{Eq8} 
\geq F_{s-1,t}(k-1-2\omega_{s,t}(k)){k-(k-1-2\omega_{s,t}(k))-1\choose \omega_{s,t}(k)}=F_{s-1,t}(k-1-2\omega_{s,t}(k)){2\omega_{s,t}(k)\choose \omega_{s,t}(k)}.
\end{equation} 
From the definition of $\omega_{s,t}(k)$, we see that $k-(s-2)t-2\leq 2\omega_{s,t}(k)$. Rewriting this inequality yields $k-1-2\omega_{s,t}(k)\leq (s-2)t+1$. We also know from \eqref{Eq10} that $s-1\leq k-1-2\omega_{s,t}(k)$. Inducting on $s$, we see from \eqref{Eq7} that \[F_{s-1,t}(k-1-2\omega_{s,t}(k))\geq \frac{2^{k-s-2\omega_{s,t}(k)}}{(k-3-2\omega_{s,t}(k))(k-4-2\omega_{s,t}(k))\cdots(k-s+1-2\omega_{s,t}(k))}\] 
\begin{equation}\label{Eq9} 
\geq\frac{2^{k-s-2\omega_{s,t}(k)}}{(k-3)(k-4)\cdots(k-s+1)}.
\end{equation} 
Because $k\leq (s-1)t+1<2(s-2)t+1$, we have $2(k-(s-2)t-2)+1<k-2$. This implies that $2\omega_{s,t}(k)+1<k-2$, so \[{2\omega_{s,t}(k)\choose\omega_{s,t}(k)}\geq \frac{2^{2\omega_{s,t}(k)}}{2\omega_{s,t}(k)+1}>\frac{2^{2\omega_{s,t}(k)}}{k-2}.\] Combining this last inequality with \eqref{Eq8} and \eqref{Eq9}, we obtain \eqref{Eq7}. 

We may now combine Theorem \ref{Thm5} with \eqref{Eq7} to find that \[A_{321}(K_{s,t}^\beta)\geq\sum_{k=s}^{(s-1)t+1}{st-k\choose t-1}\frac{2^{k-s}}{(k-2)(k-3)\cdots(k-s+1)}\geq\sum_{k=s}^{(s-1)t+1}{st-k\choose t-1}\frac{2^{k-s}}{(k-2)^{s-2}}\] \[\geq{st-(st-2(t-1))\choose t-1}\frac{2^{(st-2(t-1))-s}}{((st-2(t-1))-2)^{s-2}}={2(t-1)\choose t-1}\frac{2^{st-2(t-1)-s}}{(st-2t-4)^{s-2}}.\] Consequently, 
\[\liminf_{t\to\infty}\sqrt[t]{A_{321}(K_{s,t}^\beta)}\geq\lim_{t\to\infty}\sqrt[t]{{2(t-1)\choose t-1}\frac{2^{st-2(t-1)-s}}{(st-2t-4)^{s-2}}}=2^s.\] 

From Remark \ref{Rem2}, we know that \[A_{321}(K_{s,t}^\beta)\leq\sum_{k=s}^{(s-1)t+1}2^{k-s}{k\choose s}{st-k\choose t-1}\leq ((s-1)t+1)^s\sum_{k=s}^{(s-1)t+1}2^{k-s}{st-k\choose t-1}.\] It is straightforward to show that the maximum value of $2^{k-s}{st-k\choose t-1}$ for $s\leq k\leq (s-1)t+1$ occurs when $k=st-2(t-1)$. Therefore, \[A_{321}(K_{s,t}^\beta)\leq ((s-1)t+1)^s((s-1)t-s+2)2^{st-2(t-1)-s}{2(t-1)\choose t-1}.\] As a consequence, \[\limsup_{t\to\infty}\sqrt[t]{A_{321}(K_{s,t}^\beta)}\leq\lim_{t\to\infty}\sqrt[t]{((s-1)t+1)^s((s-1)t-s+2)2^{st-2(t-1)-s}{2(t-1)\choose t-1}}=2^s.\qedhere\]   
\end{proof}

\begin{center}
\underline{\bf 3.5 $A_{231,321}(K_{s,t}^\alpha)$}
\end{center}
Suppose $\sigma$ and $\tau$ are two patterns of length $3$. Yakoubov enumerated the linear extensions of $K_{s,t}^\beta$ that avoid both $\sigma$ and $\tau$ (see Theorem \ref{ThmY3}). She also enumerated the linear extensions of $K_{s,t}^\alpha$ avoiding these two patterns in all cases except that in which $\sigma$ and $\tau$ are the patterns $231$ and $321$. The following theorem provides a simple recurrence relation for the numbers $A_{231,321}(U_{\text{spine}=s,n}^\alpha)$. This serves to enumerate the linear extensions of $K_{s,t}^\alpha$ avoiding $231$ and $321$ because $U_{\text{spine}=s,st}^\alpha=K_{s,t}^\alpha$. 

\begin{theorem}\label{Thm7}
If $n\leq s$, then $A_{231,321}(U_{\text{spine}=s,n}^\alpha)=1$. If $n>s$, then \[A_{231,321}(U_{\text{spine}=s,n}^\alpha)=\sum_{j=1}^sA_{231,321}(U_{\text{spine}=s,n-j}^\alpha).\]\end{theorem}
\begin{proof}
The labeled poset $U_{\text{spine}=s,n}^\alpha$ only has one linear extension if $n\leq s$ because it is a totally ordered set. This linear extension is the identity permutation in $S_n$, so $A_{231,321}(U_{\text{spine}=s,n}^\alpha)=1$ if $n\leq s$. 

Suppose $n>s$, and let $\pi$ be a linear extension of $U_{\text{spine}=s,n}^\alpha$ that avoids $231$ and $321$. Let $j$ be the smallest positive integer such that $n-j$ precedes $n$ in $\pi$. We know that $n-s$ must precede $n$ in $\pi$, so $j\leq s$. Because $\pi$ avoids $231$, all of the entries to the left of $n$ in $\pi$ must be smaller than all of the entries to the right of $n$. Furthermore, the entries to the right of $n$ must be in increasing order because $\pi$ avoids $321$. This means that $\pi=\sigma n(n-j+1)(n-j+2)\cdots(n-1)$, where $\sigma$ is a linear extension of $U_{\text{spine}=s,n-j}^\alpha$ that avoids $231$ and $321$. On the other hand, if $1\leq j\leq s$ and $\tau$ is a linear extension of $U_{\text{spine}=s,n-j}^\alpha$ that avoids $231$ and $321$, then $\tau n(n-j+1)(n-j+2)\cdots(n-1)$ is a linear extension of $U_{\text{spine}=s,n}^\alpha$ that avoids $231$ and $321$. The recurrence relation stated in the theorem now follows. 
\end{proof}

Yakoubov conjectured that $A_{231,321}(K_{s,2}^\alpha)=2^{s-1}(s-1)+1$. This follows as an easy corollary of Theorem \ref{Thm7}. 

\begin{corollary}
For all $s\geq 1$, $A_{231,321}(K_{s,2}^\alpha)=2^{s-1}(s-1)+1$. 
\end{corollary}
\begin{proof}
Using Theorem \ref{Thm7}, it is straightforward to prove by induction on $\ell$ that $A_{231,321}(U_{\text{spine}=s,s+\ell}^\alpha)$ $=2^{\ell-1}(s-1)+1$ for all $\ell\in[s]$. Since $U_{\text{spine}=s,2s}^\alpha=K_{s,2}^\alpha$, the desired result follows. 
\end{proof}

\section{Concluding Remarks}
The idea to study pattern-avoiding linear extensions of partially ordered sets is quite new, so there are certainly many open problems left in this line of research. As Anderson et al. \cite{Anderson} have done, one might wish to continue this line of work by studying linear extensions of other families of partially ordered sets. Even with respect to comb posets, there are several problems left unsettled. For example, is there a reasonably nice closed formula for $A_{321}(K_{s,t}^\alpha)$ or $A_{321}(K_{s,t}^\beta)$ for general $s,t$? Even if we are not able to find a closed formula for $A_{321}(K_{s,t}^\alpha)$, it would be interesting to find the true value of the growth rate $\displaystyle{\lim_{s\to\infty}\sqrt[s]{A_{321}(K_{s,t}^\alpha)}}$ (we derived upper and lower bounds in Theorem \ref{Thm4}). Another natural direction to follow in extending this work would involve enumerating linear extensions of comb posets that avoid patterns of length $4$ or more.   

\section{Acknowledgments} 
The author would like to express his deepest gratitude to Joe Gallian for inviting him to the 2016 REU at the University of Minnesota Duluth and for reading through this article. The REU, supported by grants NSF-1358659 and NSA H98230-16-1-0026, proved to be an exceptional experience filled with great mathematics.

\end{document}